\documentclass[11pt,a4paper]{amsart}
\usepackage{amssymb,amsmath,amsthm, amsaddr}
\usepackage{commath}
\usepackage{mathtools}

\usepackage{accents}

\usepackage[utf8]{inputenc}

\newtheorem{theorem}{Theorem}[section]

\newtheorem{lemma}[theorem]{Lemma}

\newtheorem{conjecture}[theorem]{Conjecture}
\newtheorem{defi}[theorem]{Definition}
\newtheorem{prop}[theorem]{Proposition}

\renewcommand{\geq}{\geqslant}
\renewcommand{\leq}{\leqslant}
\renewcommand{\P}{\mathbb{P}}

\newcommand{\C}{\mathcal{C}}
\newcommand{\R}{\mathbb{R}}
\newcommand{\eps}{\varepsilon}

\newcommand{\EE}{\mathbb{E}}
\newcommand{\PP}{\mathbb{P}}
\renewcommand{\O}{\mathcal{O}}

\newcommand{\D}{\Delta}
\newcommand{\G}{\Gamma}
\newcommand{\g}{\gamma}
\newcommand{\CalP}{\mathcal{P}}

\renewcommand{\mod}{\mathrm{mod \,}}
\newcommand{\N}{\mathbb{N}}

\newcommand{\skel}{\mathrm{skel \,}}

\newcommand{\M}{\mathcal{M}}

\DeclareMathOperator{\E}{\mathbb{E}}

\begin{document}

\title[Longest $k$-monotone chains]{Longest $k$-monotone chains}

\author{Gergely Ambrus}

\address{
Alfr\'ed R\'enyi Institute of Mathematics}
\email[G. Ambrus]{ambrus@renyi.hu}

\date{\today}
\thanks{Research of the author was supported by NKFIH grants PD125502 and K116451 and by the Bolyai Research Scholarship of the Hungarian Academy of Sciences. }

\maketitle

\begin{abstract}
We study higher order convexity properties of random point sets in the unit square. Given $n$ uniform i.i.d random points, we derive asymptotic estimates for the maximal number of them which are in $k$-monotone position, subject to mild boundary conditions. Besides determining the order of magnitude of the expectation, we also prove strong concentration estimates. We provide a general framework that includes the previously studied cases of $k=1$ (longest increasing sequences) and $k=2$ (longest convex chains).
\end{abstract}

\section{Higher order convexity}\label{sec_convex}

Let $X_n$ be a set of $n$ uniform, independent random points in the unit square $[0,1]^2$. It is a classical and well studied problem to determine the maximal number of points forming a {\em monotone increasing chain} in $X_n$, i.e. a set of points $p_1, \ldots, p_m$ in $X_n$ so that both the $x$-coordinates and the $y$-coordinates of $(p_i)_1^m$ form an increasing sequence. This is the geometric analogue of the famous question of {\em longest increasing subsequences} in random permutations, first mentioned in 1961 by Ulam \cite{U61}, which has been studied extensively ever since (see e.g. \cite{AD99, BDJ99, R14}). Let $L^1_n$ denote the maximum number of points of $X_n$ forming a monotone increasing chain. The order of magnitude of the expectation of $L^1_n$ was determined by Hammersley about half a century ago, with the exact value of the constant in the asymptotics determined five years later by Vershik and Kerov \cite{VK77}, and independently, by Logan and Shepp\cite{LS77}:

\begin{theorem}[\cite{H72}, \cite{LS77}, \cite{VK77}]\label{thm_hamm}
  As $n \rightarrow \infty$,
  \[
  \EE L^1_n \sim 2 n^{1/2}.
  \]
\end{theorem}

This result serves as the starting point for the current research. We are going to study point sets which satisfy a more general monotonicity criteria. We start off with a basic concept.

\begin{defi}
  A set of points $p_1, \ldots, p_m$ is a {\em chain} if their $x$-coordinates form a monotone increasing sequence. The {\em length} of the chain is the cardinality of the point set, that is, $m$.
\end{defi}
Next, one may study points of the random sample $X_n$ forming a {\em convex} chain. The motivation is two-fold. On the one hand, convex polygons with vertices among a random sample have been studied extensively in the last 50 years (see e.g. the excellent survey of Bárány~\cite{B08} or the monograph of Schneider and Weil~\cite{SW08}). On the other hand, given how fruitful and far-reaching the research of the monotone increasing subsequences has been, it is a natural attempt to transfer the results to the convex analogue.

The first steps in that direction were took in our joint paper with I. Bárány~\cite{AB09}, where we studied the order of magnitude of the maximal number of points of $X_n$ forming a {\em convex chain} together with $(0,0)$ and $(1,1)$, that is, a chain whose points are in convex position. Let $L^2_n$ denote the maximal number of points of $X_n$ in a convex chain lying under the diagonal $y=x$.
\begin{theorem}[\cite{AB09}] \label{thm_conv}
There exists a positive constant $\alpha_2$, so that as $n \rightarrow \infty$,
\[
\EE L^2_n \sim \alpha_2 n^{1/3}.
\]
\end{theorem}

We also proved a limit shape result for the longest convex chains and established upper and lower estimates for $\alpha_2$. Alternative proofs to some the results are given in  \cite{A09} and \cite{A17}.

The goal of the present paper is to study the analogous questions for higher order convexity, and to describe a unified framework to the above results. Note that the properties studied above are equivalent to non-negativity of the first (monotone increasing property) and second (convexity property) `` discrete derivatives'' of the chains. Therefore, it is natural to define higher order convexity along this scheme. Eli\'as and Matou\v{s}ek introduced the following concept in order to establish Erd\H{o}s-Szekeres type results:

\begin{defi}[Eli\'as and Matou\v{s}ek, \cite{EM13}]\label{def_positive}
The $(k+1)$-tuple $(p_1, \ldots, p_{k+1})$ of distinct points in the plane is called {\em positive}, if it lies on the graph of a function whose $k$-th derivative exists everywhere, and is nowhere negative.
 The points $(p_1, \ldots, p_m)$ in the plane form a {\em $k$-monotone chain} if their $x$-coordinates are  monotone increasing, and every $(k+1)$-tuple of them is positive.
\end{defi}

Note that in the present paper, ``monotone'' will always refer to {\em monotone increasing}. It would be an alternative to use the term ``$k$-convex''. However, there are already various other concepts existing by that name, thus we stick to ``$k$-monotone''.

A second, important remark points out the difference between cases  of $k=1,2$, and larger values of $k$. In the above definition, positivity of different $(k+1)$-tuples may be demonstrated by different functions. For $k=1,2$, there exists a single monotone/convex function containing all the points on its graph. The same property was conjectured to hold also for larger values of $k$ by Eliá\v{s} and Matou\v{s}ek \cite{EM13}. However, Rote found a counterexample for $k=3$  \cite{EM13}.

An alternative but equivalent definition may be given, see  Corollary 2.3 of \cite{EM13}: a $(k+1)$-tuple is positive iff its {\em $k$th divided difference} is nonnegative, where divided differences are defined as follows. Assume $p_1, \ldots, p_n$ are points in the plane of the form $p_i = (x_i, y_i)$ (note that here, the $x$-coordinates do not necessarily form an increasing sequence). The $j$th (forward) divided difference $\Delta_j(p_i, \ldots, p_{i + j +1})$ of the $(j+1)$-tuple $p_i, \ldots, p_{i + j }$ is defined recursively by
\begin{align}
\D_0(p_i) &:= y_i \notag\\
\D_j(p_i, \ldots, p_{i + j }) &:= \frac{\D_{j-1}(p_{i+1}, \ldots, p_{i + j }) - \D_{j-1}(p_i, \ldots, p_{i + j -1})}{x_{i + j } - x_i} \label{diffdef}
\end{align}
for every $0 \leq i \leq n - j $. Note that divided differences (and, hence, positivity of a $(k+1)$-tuple) are invariant under permutations.

Divided differences are used in polynomial approximation; in particular, they provide the coefficients for the summands of Newton's interpolating polynomial:

\begin{lemma}[Newton interpolating polynomial]\label{newton_interpolation}
Let $p_1, \ldots, p_{k+1}$ be points in the plane with distinct $x$-coordinates. Assume that $p_i = (x_i, y_i)$. The unique polynomial $P(x)$ of degree $k$ whose graph contains all the points $p_i$ for $i = 1, \ldots, k+1$ may be expressed as
\begin{equation}\label{newtonpoly}
  P(x)  = \sum_{j = 0}^ k \D_j(p_1, \ldots, p_{j+1}) \prod_{i = 1}^{j} (x - x_i) \,.
\end{equation}
\end{lemma}

Divided differences are also related to higher order derivatives by the following generalization of the mean value theorem (see \cite{Phi03}, Eq. 1.33):

\begin{lemma}[Cauchy]\label{lemma_cauchy}
  Assume that the points $p_1, \ldots, p_{k+1}$ have increasing $x$-coordinates $a := x_1 <\ldots < x_{k+1} =: b$, and they lie on the graph of a function $f$ which is $k$ times differentiable everywhere on the interval $[a,b]$.
  Then there exists $\xi \in (a,b)$ so that
  \begin{equation}\label{cauchyformula}
     \D_{k} (p_1, \ldots, p_{k+1}) = \frac {f^{(k)}(\xi)}{k!} \,.
  \end{equation}
\end{lemma}

Next, we extend the definition of divided differences to multisets of points. For a point $p$, introduce the notation
\[
p^{\circ i} = \{ \underbrace{p, \ldots, p} _ {i} \},
\]
that is, the multiset of $p$ with multiplicity $i$. Assume $p = (x, f(x))$ is a point on the graph of a function $f$, which is $k$ times differentiable at $x$. In accordance with \eqref{cauchyformula}, we define the $i$th divided difference of $p^{\circ (i+1)} $ with respect to $f$ by
\begin{equation}\label{def_dmultset}
 \D_i(p^{\circ ( i+1)} ;f):= \frac{f^{(i)}(x)}{i!}
\end{equation}
for every $i \leq k$. Note that this agrees with the limit of $\D_i(\tilde{p}_1, \ldots, \tilde{p}_{i+1})$ as $\tilde{p}_1, \ldots, \tilde{p}_{i+1}$ converge to $p$ along the graph of~$f$.

By repeatedly applying \eqref{diffdef}, we may define divided differences up to order $k$ with respect to a function $f$ of any multiset of points lying on the graph of a $k$-times differentiable function $f$. Therefore, we may extend the $k$-monotonicity property to multisets of points with respect to $f$, provided that all points of multiplicity larger than 1 lie on the graph of~$f$.

From now on, we assume that all $(k+1)$-tuples of $X_n$ are in {\em $k$-general position}, that is, they do not lie on the graph of a polynomial of degree at most $k-1$. This property holds with probability 1.

Under this assumption, positivity of $(k+1)$-tuples is a {\em transitive} property:
\begin{lemma}[\cite{EM13}, Lemma 2.5]\label{lemma_transitive}
Assume the distinct points $p_1, \ldots,  p_{k+2}$ form a chain, they are in $k$-general position, and that both $(k+1)$-tuples $(p_1, \ldots, p_{k+1})$ and $(p_2, \ldots, p_{k+2})$ are positive. Then any $(k+1)$-element subset of $(p_1, \ldots,  p_{k+2})$ is positive.
\end{lemma}

In other words, the 2-coloring of $(k+1)$-tuples given by positivity/non-positivity is a transitive coloring (defined in \cite{EM13} and  \cite{FPSS12}). This will prove to be crucial in the subsequent arguments. In particular, it implies that in order to check $k$-monotonicity of a chain, it suffices to check positivity of all of its intervals (i.e. sets of consecutive points) of length $k+1$.

\section{Results} \label{sec_results}

Our goal is to determine the order of magnitude of the maximum number of points in a uniform random sample from the unit square which form a $k$-monotone chain.  For technical reasons, we also impose boundary conditions on the chain -- these conditions will ensure that a $k$-monotone chain is also $l$-monotone for every $l \leq k$. In the case $k=1$, the boundary condition simply requires the monotone chain  to start at $(0,0)$ and finish at $(1,1)$. For convex chains, the points are required to lie in the triangle below the graph of $y = x$. For general $k$, we introduce the curve
\[
\G_k = (x, x^k) , \ x \geq 0
\]
and for every $x \geq 0$, let
\begin{equation}\label{gammakdef}
 \g_k(x) = (x, x^k) \in \G_k.
\end{equation}
For the point $\gamma_k(x)$, we write
\[
\D_i(\gamma_k(x)^{\circ ( i+1)}) := \D_i(\gamma_k(x)^{\circ ( i+1)} ; x^k)= k (k-1) \ldots(k - i + 1) x^{k-i} ,
\]
that is, we consider $k$-monotonicity with respect to $f(x) = x^k$. Note that for $x = 0$ and $x=1$, $\gamma_k(x)$ is the same for every $k$; however, ambiguity is avoided by specifying the value of $k$.

The setup is the following. Fix $k \geq 1$. For any $n \geq 1$, as set before, let $X_n$ be a set of $n$ i.i.d. uniform random points in the unit square $[0,1]^2$.

\begin{defi}\label{def_mkxn}
  Let $\M^k(X_n)$ be the set of all chains $(p_1, \ldots, p_m)$  of $X_n$ so that
\begin{equation}\label{multsetchain}
(\gamma_k(0)^{\circ k}, p_1, p_2, \ldots, p_m, \gamma_k(1)^{\circ k})
\end{equation}
is a $k$-monotone chain. Furthermore, let $L^k(X_n) =:  L^k_n$ denote the maximal cardinality of elements of $\M^k(X_n)$.
\end{defi}

Thus, $L^k_n$ is a random variable defined on the space of $n$-element i.i.d. uniform samples from the square. Note that the definition of divided differences and the boundary condition at $\gamma_k(0)$ implies that up to order $k$, the $j$th divided differences of the consecutive $(j+1)$-tuples of the chain form a monotone increasing sequence, starting at 0. Therefore,
\begin{equation}\label{mkmonotone}
  \M^k(X_n) \subset \M^j(X_n)
\end{equation}
for every $j \leq k$.

We first extend Hammersley's result \cite{H72} by generalizing Theorem~\ref{thm_hamm} and Theorem~\ref{thm_conv}.

\begin{theorem}\label{thm_expectation}
For any $k \geq 1$  there exists a positive constant $\alpha_k$ so that
\[
\lim_{n \rightarrow \infty}  n^{- \frac 1 {k+1}} \, \EE L^k_n = \alpha_k \,.
\]
Furthermore, $ n^{- \frac 1 {k+1}} \, L^k(X_n) \rightarrow \alpha_k$ almost surely, as $n \rightarrow \infty$.
\end{theorem}

The exact value of the constant is not known except for the case $k=1$, where $\alpha_1 = 2$ holds. However, we may estimate it from below:

\begin{prop} \label{prop_lowerbound}
For every $k \geq 1$, $\alpha_k \geq \frac 1 6 $.
\end{prop}

By utilizing the deviation estimates of Talagrand \cite{T96}, we obtain the strong concentration property of $L^k_n$:

\begin{theorem} \label{thm_conc}
For every $k\geq 1$, and for every $\eps>0$,
\[
\PP\left(|L^k_n - \EE L^k_n| > \eps n^{\frac{1}{2 (k+1)}} \right) \leq 5 e^{ - \eps^2 / 5 \alpha_k}
\]
holds for every sufficiently large $n$. 
\end{theorem}

Finally, we conjecture that similarly to the $k=2$ case \cite{AB09}, the above stochastic concentration property leads to {\em geometric concentration}: it implies that the {\em limit shape} of longest $k$-monotone chains satisfying the boundary conditions
is $\G_k$.

\begin{conjecture} \label{thm_limshape} For any $k\geq 1$, the longest $k$-monotone chains converge in probability to $\G_k$. That is, for any $\eps >0$,
\[
\PP( \exists \textrm{ longest $k$-monotone chain with distance $>\eps$ from $\Gamma_k$}) \rightarrow 0
\]
as $n \rightarrow \infty$.
\end{conjecture}

We stated the results for $X_n$ being chosen from the unit square. As we will see in Section~\ref{sec_geom}, the boundary conditions imply that only a small fraction of the square plays a role here: all members of $\M^k(X_n)$ lie in $C_k(0,1)$, see Definition~\ref{def_ckab}. The area of the region is $1 / (k\, 2^{k - 1})$ (see \eqref{eq_Area}); therefore, switching the base domain from the square to $C_k(0,1)$ results in multiplying $\alpha_k$ by a factor of 2, without changing the order of magnitude of $\EE L^k_n$.

\section{Geometric properties}\label{sec_geom}

We start with a geometric characterization of positivity. Let $(p_1, \ldots, p_{k+1})$  be a $(k+1)$-tuple. We define its {\em sign} to be the sign of $\D_k(p_1, \ldots, p_{k+1})$.

\begin{lemma}[\cite{EM13}, Lemma 2.4.]  \label{lemma_sign}
Assume that $\CalP = (p_1, \ldots, p_{k+1})$ is a chain in $k$-general position. For any $i \in [k+1]$, let $P_i$ be the unique polynomial of degree $k-1$ containing all the points $p_j$, $j \neq i$. The $(k+1)$-tuple $\CalP$ has sign $(-1)^{k-i}$ if $p_i$ lies below the graph of $P_i$, and has sign  $(-1)^{k-i+1}$ if $p_i$ lies above the graph.
\end{lemma}

We may naturally extend the above statement for chains containing multiple points. In the next lemma, we illustrate this for the special case when the chain consists of only 3 points, with the two endpoints having multiplicity larger than 1. The same method can be applied for the general case as long as the points of multiplicity lie on the graph of a function $f$ with the necessary differentiability properties: if $p = (x, f(x))$ has multiplicity $\beta$ in the chain, then  the approximating polynomial $P$  is required to have derivatives agreeing with those of $f$ up to order $\beta - 1$  at $x$.

Below, we define $f^{(0)}(x) := f(x) $.

\begin{lemma} \label{lemma_multsign}
Let $q = (a, f(a))$ and $ \tilde{q} = (b, f(b)) $, $a<b$ be points on the graph of a function $f$, which is $(k-1)$-times differentiable in the interval $[a,b]$. Assume furthermore that $f^{(k-1)} $ does not vanish on $[a,b]$. Let $1 \leq i \leq k+1$, and denote by  $\Phi_{i,k}(a,b,f)$ the unique polynomial of degree $k-1$ which satisfies
\begin{align}\label{def_phidef}
\begin{split}
\Phi_{i,k}^{(j)}(a,b,f)(a) &= f^{(j)}(a) \textrm{ for every } 0 \leq j \leq i-2 \textrm{, and} \\
\Phi_{i,k}^{(j)}(a,b,f)(b) &= f^{(j)}(b) \textrm{ for every } 0 \leq j \leq k -i.
\end{split}
\end{align}
Assume that for the point $p \neq q, \tilde{q}$, the $(k+1)$-tuple $\CalP'=(q^{\circ (i-1)}, p, \tilde{q}^{\circ(k-i+1)})$ is a chain. Then $\CalP'$ has sign $(-1)^{k-i}$ if $p$ lies below the graph of $\Phi_{i,k}^{(j)}(a,b,f)$, and has sign  $(-1)^{k-i+1}$ if $p$ lies above the graph.
\end{lemma}

\begin{proof}
The statement follows directly from Lemma~\ref{lemma_sign} by letting the points $p_1, \ldots, p_{i-1}$ converge to $q$, and $p_{i+1}, \ldots, p_{k+1}$ converge to $\tilde{q}$, along the graph of $f$. By Lemma~\ref{lemma_cauchy} and \eqref{def_dmultset}, all the divided differences of $\CalP$ converge to the corresponding divided differences of $(q^{\circ (i-1)}, p, \tilde{q}^{\circ(k-i+1)})$. Moreover, the polynomial $P_i$ converges to $\Phi_{i,k}(a,b,f)$, which is shown by convergence of the derivatives. Therefore, the statement follows.
\end{proof}

Note that the $(k-1)$-times differentiability of $f$ on the interval $(a,b)$ is not fully used; we only prescribe the derivatives of $\Phi_{i,k}(a,b,f)$  at $a$ and $b$ up to order $i-1$ and $k-i -1$, respectively.

As a consequence, consider the special $(k+1)$-tuple of the form $(q^{\circ k}, p)$, and let $\Phi_k(a,f):= \Phi_{k, k-1}(a,b,f)$ denote the above defined polynomial.
If $k$ is even, then for any other point $p$, the $(k+1)$-tuple $(q^{\circ k}, p)$ is positive, if $p$ lies above the graph of $\Phi_k(a,f)$. If $k$ is odd, then for any other point $q = (x,y)$, the $(k+1)$-tuple $(q^{\circ k}, p)$ is positive, if $x <a$ and $p$ lies below the graph of $\Phi_k(a,f)$, or if $x>a$, and $p$ lies above the graph of $\Phi_k(a,f)$.

Next, we apply Lemma~\ref{lemma_multsign} for the multisets of the form $(\gamma_k(a)^{\circ k}, p, \gamma_k(b)^{\circ k})$, where $0 \leq a < b$, and $k \geq 1$, see \eqref{gammakdef}, \eqref{multsetchain}. Some notations are in order. Let $a,b \geq 0$. Define the polynomials
\begin{align}
\Phi_k(a) (x) &= x^k - (x-a)^k  \label{eq_phi}\\
\Psi_k(a,b) (x) &= x^k - (x -a)^{k-1} (x-b). \label{eq_psi}
\end{align}
With a slight abuse of notation, we are going to denote the graphs of these polynomials by the same symbols. It will always be clear from the context which meaning do we refer to.

\begin{defi}\label{def_ckab}
For $0 \leq a< b$, the cell $C_k(a,b)$ is defined as follows:
\begin{description}
  \item[\rm For $k=1$] $C_1(a, b) = [a,b]^2$, the square with diagonal vertices $\g_1(a)$ and $\g_1(b)$;
  \item[\rm For $k =2$] $C_2(a, b)$ is the triangle bounded by $\Phi_k(a), \Phi_k(b)$, and $\Psi_k(a,b) =  \Psi_k(b,a)$;
  \item[\rm For $k \geq 3$] $C_k(a, b)$ is the 4-vertex cell bounded by $\Phi_k(a), \Phi_k(b), \Psi_k(a,b)$ and $\Psi_k(b,a)$.
\end{description}
\end{defi}

\begin{figure}[h]
\centering
  \includegraphics[width=0.6\textwidth]{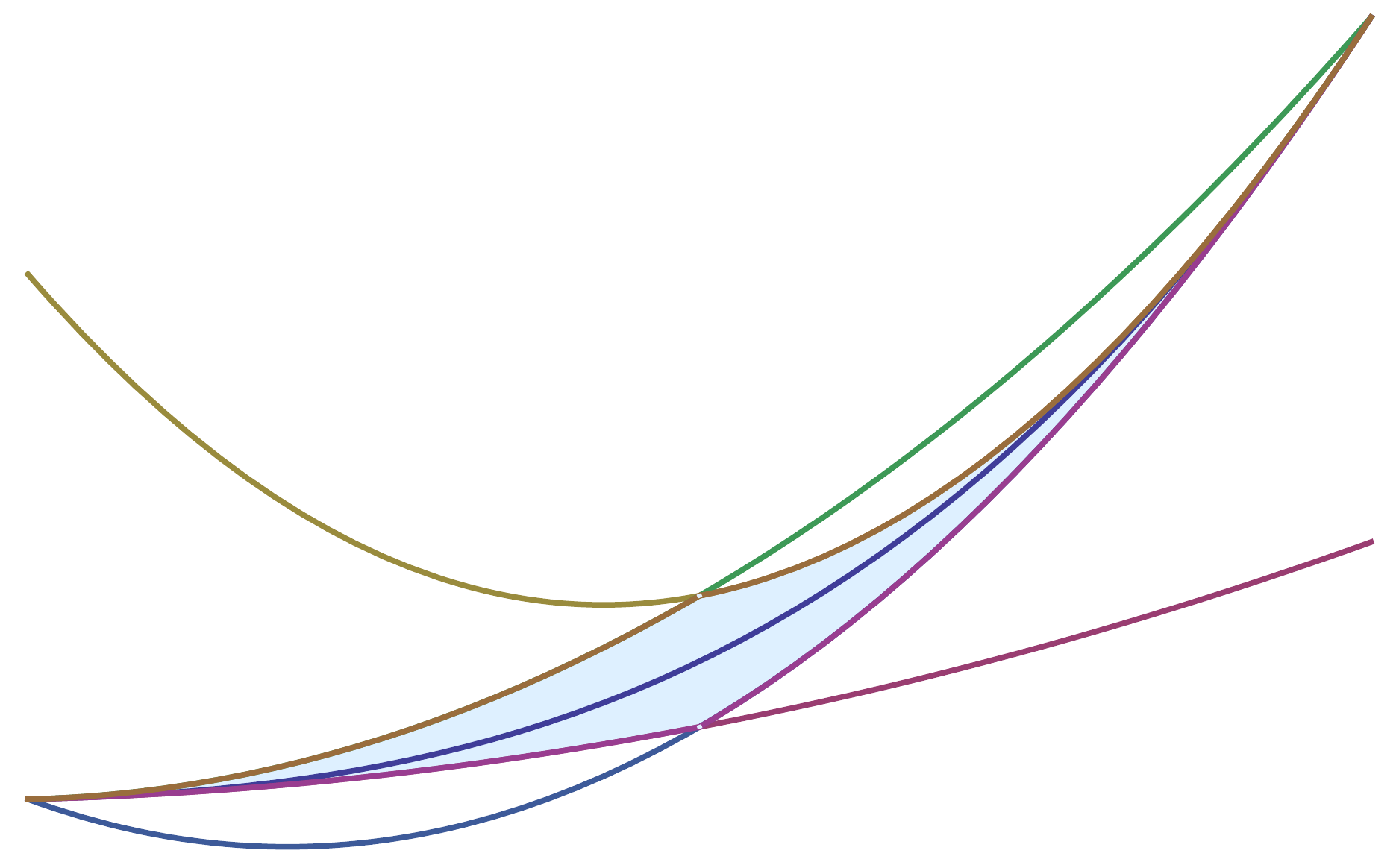}\\
  \caption{The cell $C_k(a,b)$ between $\g_k(a)$ (left endpoint) and $\g_k(b)$ (right endpoint), $k$ odd, shaded with blue.  The curve $\G_k$ runs in the middle of the cell. The lower boundary consists of $\max \{\Phi_k(a),\Psi_k(b,a)\}$, the upper boundary is defined by
$\min \{ \Psi_k(a,b), \Phi_k(b) \}$ .}\label{fig_Cab}
\end{figure}

Let us elaborate on the $k \geq 2 $ case (see Figure~\ref{fig_Cab}) . When $k$ is even, the lower boundary of $C_k(a,b)$ consists of two arcs: $\Phi_k(a)$ for $a \leq x \leq (a+b) /2$, and $\Phi_k(b)$ for $(a+b) /2 \leq x \leq b$. The upper boundary again consists of two arcs: $\Psi_k(a,b)$ for $a \leq x \leq (a+b) /2$, and $\Psi_k(b,a)$ for $(a+b) /2 \leq x \leq b$ (for $k=2$, these coincide with each other). When $k$ is odd, the lower and upper boundaries of $C_k(a,b)$ are the same as in the even case; however, for $(a+b) /2 \leq x \leq b$,  $\Phi_k(b)$ is the upper boundary, while $\Psi_k(b,a)$ is the lower boundary. Thus, in any case, $\g_k(a)$ and $\g_k(b)$ are two opposite vertices of $C_k(a,b)$, and other two vertices both have $x$-coordinates $(a+b)/2$.

The importance of the cell $C_k(a,b)$ is given by the following statement.

\begin{lemma}\label{lemma_cell}
Let $k \geq 1$, and $0 \leq a < b$. The set of points $p$ in the plane for which
\[
(\gamma_k(a)^{\circ k}, p, \gamma_k(b)^{\circ k})
\]
is a $k$-monotone chain is exactly $C_k(a,b)$.
\end{lemma}

\begin{proof}
We prove the statement first assuming $a>0$. The case $a=0$ may be obtained  by a standard limit argument.

Let $\Phi_{i,k}(a,b,x^k)$ be the polynomial defined by \eqref{def_phidef} for $f(x) = x^k$. By comparing derivatives, we obtain that for every $1 \leq i \leq k+1$,
\begin{equation}\label{eq_phiabk}
 \Phi_{i,k}(a,b,x^k) = x^k - (x-a)^{i-1}(x-b)^{k+1 - i }.
\end{equation}

By  Lemma~\ref{lemma_multsign}, our goal is to determine the intersection of the regions above $\Phi_{i,k}(a,b,x^k)$ for $i = k+1, k-1, k-3, \ldots, 1 +\textrm{mod}(k, 2)$ and the regions below $\Phi_{i,k}(a,b,x^k)$ for $i = k, k-2, k-4, \ldots, 2 -\textrm{mod}(k, 2)$ (see Figure~\ref{fig_Fiab}). Thus, we have to determine
\[
\max_{ \substack { i \equiv k+ 1 (\mod 2)\\ 1 \leq i \leq k+1}}\Phi_{i,k}(a,b,x^k)
\]
and
\[
\min_{ \substack { i \equiv k (\mod 2)\\ 1 \leq i \leq k+1}}\Phi_{i,k}(a,b,x^k)
\]
for every $x \in [a,b]$. By \eqref{eq_phiabk}, we obtain that for $x \in [a, (a+b)/2]$, the above extrema occur when $i = k+1$, and $i =k$, respectively, while for $x \in [(a+b)/2,b]$, the extrema are taken when $i = 1,2$. Therefore, the boundary of $C_k(a,b)$ is constituted by the polynomials of the form \eqref{eq_phi} and \eqref{eq_psi}.
\end{proof}

\begin{figure}[h]
\centering
  \includegraphics[width=0.6\textwidth]{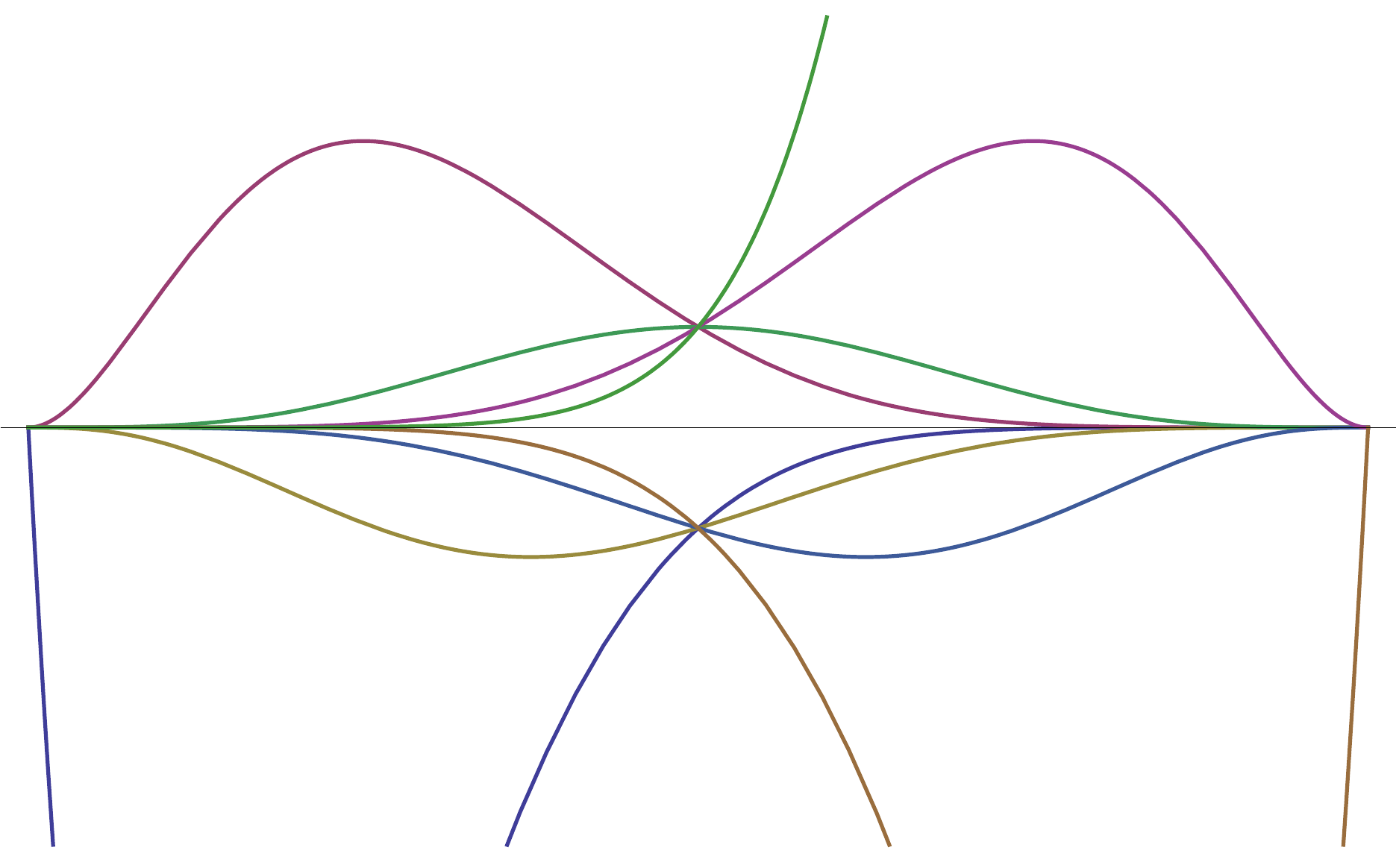}\\
  \caption{Graphs of the polynomials $x^k - \Phi_{i,k}(a,b,x^k)$, $1 \leq i \leq k+1$, plotted between $a$ and $b$, in the $k = 7$ case. }\label{fig_Fiab}
\end{figure}

The next lemma states that not only $C_k(a,b)$ is the location of the $k$-monotone chains, but given two $k$-monotone chains in neighboring cells, we may concatenate them.

\begin{lemma}\label{lemma_concatenate}
  Assume $0 \leq a < b < c$, and that the points $p_1, \ldots, p_l$ and $p_{l+1}, \ldots,  p_m$ are so that $(\gamma_k(a)^{\circ k}, p_1, \ldots, p_l, \gamma_k(b)^{\circ k})$ and $(\gamma_k(b)^{\circ k}, p_{l+1}, \ldots, p_m, \gamma_k(c)^{\circ k})$ are $k$-monotone chains in $C_k(a,b)$ and $C_{b,c}$, respectively. Then
  \[
  (\gamma_k(a)^{\circ k}, p_1, \ldots, p_m, \gamma_k(c)^{\circ k})
  \]
is  a $k$-monotone chain in $C_k(a,c)$.
\end{lemma}

\begin{proof}
By the remark following Lemma~\ref{lemma_transitive},
\[
(\gamma_k(a)^{\circ k}, p_1, \ldots, p_l,\gamma_k(b)^{\circ k},p_{l+1}, \ldots,  p_m, \gamma_k(c)^{\circ k})
\]
is a $k$-monotone chain. By the transitivity property provided by Lemma~\ref{lemma_transitive}, we may delete from this chain any point, still maintaining $k$-monotonicity. Therefore,
\[
(\gamma_k(a)^{\circ k}, p_1, \ldots, p_l,\gamma_k(b)^{\circ (k-1)},p_{l+1}, \ldots,  p_m, \gamma_k(c)^{\circ k})
\]
is also $k$-monotone. By iterating the erasure process, we finally erase all copies  of $\gamma_k(b)$, yielding the statement.
\end{proof}

Next, we introduce a transformation mapping $C_k(a,b)$ to $C_k(c,d)$ which preserves $k$-monotonicity, showing the equivalence of the cells $C_k(a,b)$ with respect to problems regarding $k$-monotone chains.

Let $0<a<b$ and $0<c<d$, and define the transformation $T_{a,b,c,d}: \R^2 \rightarrow \R^2$ by
\begin{equation}\label{eq_transfo}
  T_{a,b,c,d}(x,y)= \left(c + (x - a ) \frac {d-c}{b-a}, \left(c + (x-a) \frac {d-c}{b-a} \right)^k + (y - x^k) \left( \frac {d-c}{b-a}\right)^k \right).
\end{equation}

\begin{lemma}\label{lemma_transfo}
For any $0<a<b$ and $0<c<d$, the map $T_{a,b,c,d}$ preserves $k$-monotonicity, keeps $\Gamma_k$ fixed, and maps the uniform distribution on $C_k(a,b)$ onto the uniform distribution on $C_k(c,d)$.
\end{lemma}

\begin{proof}
Notice that
\[
T_{a,b,c,d}(a + t(b-a), (a + t(b-a))^k + \tilde{y}) = \left(c + t (d-c), (c + t (d-c))^k + \tilde{y} \left( \frac {d-c}{b-a}\right)^k \right).
\]
Thus, it is immediate that $T_{a,b,c,d} (\G_k) = \G_k$ (as this corresponds to the case $\tilde{y} =0$). On  the boundary of $C_k(a,b)$, $\tilde{y}$ is either $\pm(t(b-a))^k$, $\pm((1-t)(b-a))^k$, $\pm t^{k-1}(1 - t)(b-a)^k$, or $\pm t (1 - t)^{k-1}(b-a)^k$, which are mapped to the same expressions with $(d-c)$ in place of $(b-a)$. Therefore, $T_{a,b,c,d}$ maps $C_k(a,b)$ onto $C_k(c,d)$ by mapping the boundary curves to the corresponding ones. Measure invariance is seen by writing
\[
T_{a,b,c,d} = G_k^{-1} \circ A \circ G_k,
\]
where $G_k: (x,y) \mapsto (x, y-x^k)$, and $A$ is an affine map.  Thus, it only remains to check the invariance of $k$-monotonicity under $T_{a,b,c,d}$. To this end, we may write $T_{a,b,c,d}$ as the composition of three maps:
\[
T_{a,b,c,d} =T_{0, d-c, c, d} \circ T_{0, b-a, 0, d-c}\circ T_{a,b,0,b-a}.
\]
Assume $f(x)$ has the form $f(x) = x^k + g(x)$, then $f^{(k)}(x) = k! + g^{(k)}(x)$. The first and third of the above maps do keep this derivative fixed, as $g(x)$ is preserved by them.  Finally, the map $ T_{0, b-a, 0, d-c}$ is a linear map scaling the $x$ and $y$ coordinates independently, therefore, it preserves the sign of the derivatives. Therefore, using Definition~\ref{def_positive}, $T_{a,b,c,d}$ preserves positivity of $(k+1)$-tuples, hence it also preserves positivity.
\end{proof}

We conclude this section by calculating the area of the base cell $C_k(a,b)$. By \eqref{eq_phi}, \eqref{eq_psi} and the discussion afterwards,  the distance between the upper and lower boundary of $C_k(a,b)$ is $(x-a)^{k-1}(b-x)$. Therefore,
\begin{align}\label{eq_Area}
\begin{split}
A(C_k(a,b)) &=  \int_a^{(a+b)/2} (x - a)^{k-1} (b - a)dt + \int_{(a+b)/2}^b (b - x)^{k-1} (b - a)dt \\
    &= \frac {(b - a)^{k+1}}{k \, 2^{k-1}}.
\end{split}
\end{align}

\section{The Poisson model}

In order to make the problem more approachable, in this section we switch to the {\em Poisson model}, that has by now became an industry standard (see e.g. \cite{R14}).

Let $\Pi$ be a planar homogeneous Poisson process  with intensity 1. Given any domain $D$ of area $A(D)$ in the plane, the number of points of $\Pi$ in $D$ has Poisson distribution with parameter $A(D)$. That is, its probability mass function is given by
\begin{equation}\label{poissonmassfn}
  \mathbb{P} (|D \cap Pi| = k) = \frac{A(D)^k e^{- A(D)}}{k!} \,,
\end{equation}
and its expectation is $A(D)$.

We are going to use the following standard tail estimate for Poisson random variables, see Proposition 1 of~\cite{G87}. Assume that $X$ has Poisson distribution with parameter $\lambda$. Then
\begin{equation}\label{poissontail}
  \P (X \geq m) \leq \frac{m+1}{m+1 - \lambda}\, \P(X = k) = \frac{m+1}{m+1 - \lambda} e^{- \lambda} \, \frac{\lambda^m}{m!} \,.
\end{equation}

For arbitrary $0 \leq a<b$, let $N^k_\Pi(a,b)$ be the cardinality of $\Pi \cap C_k(a,b)$. By \eqref{eq_Area} and \eqref{poissonmassfn}, $N^k_\Pi(a,b)$ has a Poisson distribution with parameter (and mean) $(b - a)^{k+1} / (k \, 2^{k-1})$.
Moreover, conditioning on the event $N^k_\Pi(a,b) = N$, the joint distribution of the points of $\Pi$ falling in $C_k(a,b)$ is the same as the joint distribution of $N$ i.i.d. uniform points in $C_k(a,b)$.

As the analogue of Definition~\ref{def_mkxn}, we introduce

\begin{defi}\label{def_mkpi}
   Let $\M^k _\Pi(a,b)$ be the set of all chains $(p_1, \ldots, p_m) \subset \Pi \cap C_k(a,b) $ so that
\begin{equation}\label{gkab}
(\gamma_k(a)^{\circ k}, p_1, p_2, \ldots, p_m, \gamma_k(b)^{\circ k})
\end{equation}
is a $k$-monotone chain. Furthermore, let $L^k(a,b)$ denote the maximal cardinality of elements of $\M^k_\Pi(a,b)$.
\end{defi}
By Lemma~\ref{lemma_transfo} and the invariance property of the Poisson process, the distribution of $L^k(a, b)$ depends solely on $(b -a)$; therefore, the results below involving $L^k(0,n)$ remain also valid for the general variables $L^k(a, b)$.

Next, we establish the link between the Poisson and the uniform models. By \eqref{eq_Area}, the area of $C_k(0,n)$ is $n^{k+1}/ (k 2^{k-1})$, therefore, $N^k_\Pi(0,n)$ has a Poisson distribution with parameter $n^{k+1}/ (k 2^{k-1})$. On the other hand, let us denote by $N^k_n$ the number of points of $X_n$ in $C_k(0,1)$. Then, $N^k_n$ has binomial distribution with parameters $n$ and $1/(k 2^{k-1})$, and its mean is $n / (k 2^{k-1})$. Standard Chernoff type concentration estimates for binomial and Poisson random variables (see e.g. Chapter 2 of \cite{BLM13}) yield the following quantitative bound.

\begin{prop}\label{poissonuniform} For any $k\geq 1$, and for any $c>0$,
  \[
  \P\left( | N^k_\Pi(0,n) - N^k_{n^{k+1}}| > c \, \sqrt{\frac{n^{k+1}}{k 2^{k-1}}}\right) < 4 e^{-c^2 /3} \]
holds for every sufficiently large $n$.
\end{prop}

This also implies that the random variable $L^k(0,n)$  is a good approximation of $L^k_{n^{k+1}}$. Applying Proposition~\ref{poissonuniform} with $c = \eps n^{(k+1)/2}$ allows us to transfer the statement of Theorem~\ref{thm_expectation} to the Poisson model. We are going to prove the following theorem, which readily implies Theorem~\ref{thm_expectation}.

\begin{theorem}\label{thm_expectation_poisson}
For any $k \geq 1$,  there exists a positive constant $\alpha_k$ so that
\begin{equation}\label{elk0nas}
  \lim_{n \rightarrow \infty}  n^{-1} \, \EE L^k(0,n) = \alpha_k \,.
\end{equation}

Furthermore, $ n^{-1} \, L^k(0,n) \rightarrow \alpha_k$ almost surely, as $n \rightarrow \infty$.
\end{theorem}

First, we need an upper bound on $\E L^k(0,n)$. In the uniform model for $k=2$, this is fairly easy to establish. The probability that a random chain of length $n$ is convex may be calculated exactly~\cite{BRSZ00}, based on a beautiful argument of Valtr~\cite{V95}. The probability that $n$ uniform independent random points in the unit square form a convex chain is exactly
\[
\frac{1}{n! (n+1)!}\,.
\]
The calculation is based on rearranging convex chains while keeping the underlying probability space invariant. Unfortunately, this approach brakes down for larger values of $k$, and thus, such a sharp result does not hold in the more general setting. However, we may still prove that the probability of the existence of very long $k$-monotone chains is minuscule.

\begin{lemma}\label{lemma_expsmall}
  For every $k \geq 1$ there exists a constant $c_k$ so that
  \begin{equation}\label{elkckn}
       \E L^k(0,n) < c_k n
  \end{equation}
holds for every $n \geq 1$.
\end{lemma}

\begin{proof}
As the statement is known for the cases $k =1$ and $k=2$, we may assume that $k \geq 3$. Also, it is sufficient to prove $\eqref{elkckn}$ for sufficiently large values of $n$. Let $C$ be a constant whose value we are going to specify later. Set $N = n^{k+1}$. As we noted before, $N^k_\Pi(0,n)$ has Poisson distribution with parameter $n^{k+1} / (k 2^{k-1})$. Therefore, using~\eqref{poissontail},


\begin{align}\label{elkn}
\begin{split}
\E L^k(0,n) &= \int_0^\infty \P(L^k(0,n) \geq x) dx \\
&\leq C n + \int_{Cn}^N \P(L^k(0,n) \geq x) dx + \int_N^\infty \P(L^k(0,n) \geq x)dx \\
& \leq C n + N \, \P (L^k(0,n) \geq C n )  + \int_N^\infty \P(N^k_\Pi(0,n) \geq x)dx \\
& \leq C n +  N \, \P (L^k(0,n) \geq C n ) + 2 \sum_{ i =N}^\infty \P (N^k_\Pi (0,n) = i)\\
& \leq C n +  N \, \P (L^k(0,n) \geq C n ) + \frac {1}{4^N}\,.
\end{split}
\end{align}
Thus, it suffices to show that for suitably large $C$ (depending on $k$ only),
\[
\P (L^k(0,n) \geq C n ) = o(n^{-k}).
\]

Call a $k$-monotone chain in $\Pi \cap C_k(0,n)$ {\em long} if its cardinality is at least $C n$. To every such long $k$-monotone chain $\C$ we assign its {\em skeleton} as follows. Assume that $ \C = \{p_1, \ldots, p_m\}$ (with the points ordered according to their $x$-coordinates), where $m \geq Cn $. Then
\begin{align}\label{def_skeleton}
\begin{split}
\skel(\C) &= \{ \gamma_k(0)^{\circ k }, p_{\lfloor \frac m n \rfloor}, p_{2 \lfloor \frac m  n \rfloor}, \ldots, p_{n-1 \lfloor \frac m  n \rfloor} ,\gamma_k(n)^{\circ k } \} \\
&=: \{\gamma_k(0)^{\circ k }, s_1, \ldots, s_{n-1}, \gamma_k(n)^{\circ k }  \}\,,
\end{split}
\end{align}
that is, $s_i = p_{i \lfloor m / n \rfloor}$ for every $i = 1, \dots, n-1$. Also, set $s_i = \gamma_k(0)$ for $i \leq 0$ and $s_j = \gamma_k(n)$ for $j \geq n$. Any long chain is cut into $n$ intervals of length at least $C$ by its skeleton.

The free part of the skeleton is a chain of length $n-1$ contained in $C_k(0,n)$. The  distribution of the long chains in $ \Pi \cap C_k(0,n)$ induces a probability distribution $\mu$ on the space of skeletons. By the law of total probability,
\[
\P(L^k(0, n) \geq C n) = \int \P (\exists \textrm{ a long }k\textrm{-monotone chain } \C \ | \ \skel(\C) = S) d \mu (S),
\]
where the integral is taken over the space of possible skeleta. Thus, \eqref{elkn}, implies \eqref{elkckn} as long as
\[
\P (\exists \textrm{ a long }k\textrm{-monotone chain } \C \ | \ \skel(\C) =S) < o(n^{-k})
\]
holds true for every possible skeleton $S$, with the constants of the asymptotic estimate being independent of $S$. This is what we are going to prove.

Let us now fix $S$ of the form \eqref{def_skeleton} and assume that $\C$ is a long $k$-monotone chain with $\skel(\C) = S$. Let $p \in \C \setminus \skel(\C)$. For any point $u \in \R^2$, let $x(u)$ denote its $x$-coordinate. There exists a unique index $i$ so that $x(p) \in [x(s_i), x(s_{i+1})]$. Then, by Definition~\ref{def_positive} of $k$-monotone chains, the $(k+1)$-tuples
\[
(s_{i - k +1}, s_{i - k + 2}, \ldots, s_i, p)
\]
and
\[
(s_{i - k +2}, s_{i - k + 3}, \ldots, s_i, p, s_{i+1})
\]
are positive. Let $P_1$ be the unique polynomial of degree $k-1$ whose graph contains the points $s_{i - k +1}, s_{i - k + 2}, \ldots, s_i$, and similarly, let $P_2$ be the unique polynomial of degree $k-1$ whose graph contains the points $s_{i - k +2}, s_{i - k + 3}, \ldots, s_i, s_{i+1}$, possibly using the extended definition for multisets discussed in Section~\ref{sec_convex}. That is, if $\gamma_k(0)$ appears with multiplicity $\beta$ among the nodes for $P_i$ for $i = 1$ or 2, than the derivatives up to order $\beta - 1$ of $P_i$ at $0$ are required to agree with those of $x^k$ at $0$.

Lemma~\ref{lemma_sign} and its generalization to multisets implies that the point $p$ lies in the region $R_i$ bounded by the graphs of the polynomials $P_1$ and $P_2$ over the interval $[x(s_i), x(s_{i+1})]$. Lemma~\ref{newton_interpolation} and formula \eqref{newtonpoly} shows that
\begin{align*}
\begin{split}
P_1(x) - P_2(x) &=  \Big( \D_{k-1}(s_{i - k +1}, s_{i - k + 2}, \ldots, s_i)  \\
 & \quad- \D_{k-1}(s_{i - k +2}, s_{i - k + 3}, \ldots, s_i, s_{i+1}) \Big)  \prod_{ j = 1}^{k-1} (x - x(s_{i +1 - j}))\,.
\end{split}
\end{align*}
Therefore,
\begin{align}\label{Ari}
\begin{split}
A(R_i) &= \int_{x(s_i)}^{x(s_{i+1})} |P_1(x) - P_2(x)| \\
& \leq \big(x(s_{i+1}) - x(s_i)\big)\big(x(s_{i+1}) - x(s_{i - k +2})\big)^{k-1} D_i \\
& \leq \big(x(s_{i+1}) - x(s_{i - k +2})\big)^ k D_i
\end{split}
\end{align}
with
\[
D_i = \D_{k-1}(s_{i - k +2}, s_{i - k + 3}, \ldots, s_i, s_{i+1}) - \D_{k-1}(s_{i - k +1}, s_{i - k + 2}, \ldots, s_i) \,.
\]

Since
\[
\sum_{i = 0}^{n-1}  x(s_{i+1}) - x(s_{i - k +2})
= \sum_{j=0}^{k-2} x(s_{n -j }) - x(s_{-j})  < k n,
\]
there are at least $(2n) /3$ indices $i$ in the interval $[0, n-1]$ so that
\begin{equation}\label{xint3k}
  x(s_{i+1}) - x(s_{i - k +2}) \leq 3k \,.
\end{equation}
On the other hand, the $k$-monotonicity of $\C$ implies that  $(\D_{k-1}(s_{i - k +1}, s_{i - k + 2}, \ldots, s_i)) _{i = 0 }^{n+k-1}$ is a monotone increasing sequence, which, by \eqref{def_dmultset}, satisfies
\[
\D_{k-1}(s_{ - k +1}, s_{ - k + 2}, \ldots, s_0) = 0
\]
and
\[
\D_{k-1}(s_{ n}, s_{ n+1}, \ldots, s_{n+ k -1}) = k n \,.
\]
Thus, there exist at least $(2n) /3$ indices $j$ in the interval $[0, n-1]$ so that
\begin{equation}\label{dj3k}
  \D_{k-1}(s_{j - k +1}, s_{j - k + 2}, \ldots, s_j) \leq 3 k.
\end{equation}
Combining \eqref{xint3k} and \eqref{dj3k} with \eqref{Ari}, we obtain that there at least $n/3$ indices $i \in [0, n-1]$ so that
\[
A(R_i) \leq (3 k)^{k+1}.
\]
In order for $\C$ to be long, each of these regions must contain at least $C$ points of $\Pi$. Pick such a region $R$. By \eqref{poissontail} and Stirling's approximation,
\[
\P\left( |R \cap \Pi| \geq C \right) \leq 2 \left( \frac{e (3k)^{k+1}}{C}\right)^C
\]
holds for any sufficiently large $C$. Therefore, for any given $\eps>0$, there exists a corresponding $C$ so that the above probability is bounded above by $\eps$. For that choice of $C$,
\begin{align*}
\P (\exists \textrm{ a long }k\textrm{-monotone chain } \C \ | \ \skel(\C) = S) &\leq \prod_{i = 0}^{n-1} \P\left( |R_i \cap \Pi| \geq C \right)\\
&\leq \eps^{n/3},
\end{align*}
where the independence property of the Poisson process is used in the first inequality.
The proof is finished by noting that he above expression is of order $o(n^{-k})$ for sufficiently small values of ~$\eps$, and all the above estimates depend on $k$ only.
\end{proof}

\section{Expectation and concentration estimates} \label{sec_expectation}

In this section, we show that the order of magnitude of the length of the longest $k$-monotone chains among $n$ random points is $n^{1/(k+1)}$. We are going to prove this in the Poisson model, Theorem~\ref{thm_expectation_poisson}, which implies  Theorem~\ref{thm_expectation} of the uniform model. The proof builds on Kingman's subbaditive ergodic theorem. Below, we present a version of it along with an important extension by Liggett.

\begin{theorem}[Kingman's subadditive ergodic theorem with Liggett's extension \cite{K73, L85}] \label{thm_kingman}
Assume $X_{n,m}$, $n,m \in \N$, is a family of random variables satisfying the following conditions:
\begin{itemize}
\item[S1)] $X_{l,n} \leq X_{l,m} + X_{m,n}$ whenever $0 \leq l < m < n$;
\item[S2)] For every $s \geq 0$ integer, the joint distributions of the process $\{X_{m+s,n+s}\}$ are the same as those of  $\{X_{m,n}\}$;
\item[S3)] For each $n$, $\E |X_{0,n}| < \infty$ and $\E X_{0,n} > -c n$ for some constant $c$.
\end{itemize}
Then
$$
\gamma = \lim_{n \rightarrow \infty}\frac {\E X_{0,n}}{n}
$$
exists,
$$
X = \lim_{n \rightarrow \infty}\frac {X_{0,n}}{n}
$$
exists almost surely, and
$
\E X = \gamma.
$

Furthermore, if the stationary sequences $(X_{in, (i+1)m})_{i=1}^\infty$ are ergodic for any $m \geq 1$, then $X = \gamma$ almost surely.
\end{theorem}

\begin{proof}[Proof of Theorem~\ref{thm_expectation_poisson}]

We show that Conditions S1), S2) and S3) of Theorem~\ref{thm_kingman} hold for the family of random variables $X_{m,n}:=-L^k(n,m)$, $n,m \in \N$, see Definition~\ref{def_mkpi}. Lemma~\ref{lemma_concatenate} shows that
\[
L^k(a,c) \geq L^k(a,b) + L^k(b,c)
\]
for every $0 \leq a< b< c$, showing the validity of S1). The invariance property S2) follows from Lemma~\ref{lemma_transfo}. Finally, S3) is implied by Lemma~\ref{lemma_expsmall}.

Therefore, we may apply Theorem~\ref{thm_kingman} to obtain that $\E L^k(0,n) \approx \alpha_k n$ with some positive constant $\beta_k$. Moreover, $(L^k(in, (i+1)n)_{i=1}^\infty$ is a sequence of independent, identically distributed random variables, hence it is ergodic. Therefore, $n^{-1}L^k(0,n)$ converges to $\alpha_k $ almost surely.
\end{proof}

Theorem~\ref{thm_expectation_poisson} and Proposition~\ref{poissonuniform} implies that
\[
 \frac {L^k_n}{n^{1/(k+1)}} \rightarrow \alpha_k
\]
almost surely, proving Theorem~\ref{thm_expectation}.

Next, we derive a lower bound on the constant $\alpha_k$ of \eqref{elk0nas}.

\begin{proof}[Proof of Proposition~\ref{prop_lowerbound}]
We are going to prove the statement in the Poisson model by showing that for sufficiently large $n$,
\[
\E L^k(0 ,n ) \geq \frac {n }{6}
\]
holds.

Set $a_i = 3i$ for every $i \in [0, \lfloor n / 3 \rfloor ]$.
By \eqref{eq_Area}, the area of $C_k(a_i, a_{i+1})$ (see Definition~\ref{def_ckab}) is
\[
A(C_k(a_i, a_{i+1}))= \frac {3^k}{k 2^{k-1}} >1.
\]
Since the number of points of $\Pi$ in $C_k(a_i, a_{i+1})$ has Poisson distribution with parameter $A(C_k(a_i, a_{i+1}))$,
\[
\P(|\Pi \cap C_k(a_i, a_{i+1})| =0) = e^{-A(C_k(a_i, a_{i+1})) } < \frac 1 e \,.
\]
Let $Y$ be the number of cells of the form $C_k(a_i, a_{i+1})$ in which $\Pi$ has at least one point. Then $Y \sim B( \lfloor \frac n 3 \rfloor, p)$ with $p > 1 - 1/e > 1/2$. Let $\lambda := \E Y$, then $\lambda > n (1 - 1/e)/3$. By a standard Chernoff-type bound for binomial random variables (see Theorem A.1.12 of \cite{AS00}),
\[
\P\left(Y \leq \lambda- c\sqrt{\lambda\log \lambda} \right) < \lambda^{-c^2/2}.
\]
Therefore, for sufficiently large $n$, 
\begin{equation}\label{pyn6}
\P( Y > n /6) \approx 1.
\end{equation}

Let us now take a point $p$ of $\Pi$ in each of the non-empty cells, and let $\mathcal{S} = \{s_1, \ldots, s_Y\}$ be the collection of these points ordered with respect to their $x$-coordinates. By the construction,
\[
\left( \gamma_k(0)^{\circ k}, s_1, \gamma_k(a_1)^{\circ k}, s_2 , \ldots, \gamma_k(a_{\lfloor n / 3 \rfloor})^{\circ k}, \gamma_k(n)^{\circ k} \right)
\]
is a $k$-monotone chain, where each $s_i$ is placed in its corresponding interval so that we obtain a chain.
By repeatedly applying Lemma~\ref{lemma_concatenate}, we deduce that 
\[
\left( \gamma_k(0)^{\circ k}, s_1, s_2, \ldots, s_Y, \gamma_k(n)^{\circ k} \right)
\]
is also a $k$-monotone chain. Therefore, $L^k(0,n) \geq Y$. The proof is finished by referring to~\eqref{pyn6}, which shows that $\E L^k (0,n) \geq n/6$.
\end{proof}

We finish this section by establishing the exponential concentration estimate for $L^k_n$. Theorem~\ref{thm_conc} is straightforward consequence of Talagrand's strong concentration inequality.

\begin{theorem}[Talagrand \cite{T96}]\label{thm_talagrand}
Suppose $Y$ is a real-valued random variable on a product
probability space $\Omega^{\otimes n}$, and that $Y$ is 1-Lipschitz
with respect to the Hamming distance, meaning that
$$ | Y(x)-Y(y)| \leq 1$$
whenever $x$ and $y$ differ in one coordinates. Moreover assume that
$Y$ is {\sl $f$-certifiable}. This means that there exists a
function $f: \mathbb{N} \rightarrow \mathbb{N}$ with the following
property: for every $x$ and $b$ with $Y(x) \geq b$ there exists an
index set $I$ of at most $f(b)$ elements, such that $Y(y) \geq b$
holds for every $y$ agreeing with $x$ on $I$. Let $m$ denote the
median of $Y$. Then for every $s>0$ we have
\\
$$ \PP (Y \leq m-s)\leq 2 \, \mathrm{exp}\left(\frac{-s^2}{4f(m) } \right)$$
and
$$ \PP (Y \geq m+s)\leq 2 \, \mathrm{exp}\left(\frac{-s^2}{4f(m+s) } \right).$$
\end{theorem}

The conditions of Theorem~\ref{thm_talagrand} are clearly satisfied by the random variable $L^k_n$ with the certificate function $f(b) = b$, by fixing the points of the longest $k$-monotone chain in  $X_n$. Since $L^k_n \leq n$, exponential concentration ensures that the mean and the median are within a distance of $O(n^{1/2(k+1)})$ of each other. Thus, in the above estimates, $m \approx \alpha_k n^{1/(k+1)}$, and  setting $ s = \eps n^{1/2(k+1)}$, we obtain Theorem~\ref{thm_conc}.

The same proof yields the analogous concentration estimate for $L^k(0,n)$:

\begin{theorem} \label{thm_conc_poisson}
For every $k\geq 1$, and for every $\eps>0$,
\[
\PP\left(|L^k(0,n) - \EE L^k(0,n)| > \eps \sqrt{n} \right) \leq 5 e^{- \eps^2/5 \alpha_k}.
\]
holds for every sufficiently large $n$.
\end{theorem}

Summarizing the results proved in this section, we showed that $L^k_n$ is a random variable exponentially concentrated in a neighbourhood of radius $\O(n^{1/2(k+1)})$ around its mean, which converges to $\alpha_k n^{1/(k+1)}$.

\end{document}